\documentclass[12pt]{article}
\usepackage{amsmath,amsfonts,amssymb,amsthm,bbm}
\usepackage[ruled,nofillcomment]{algorithm2e}
\usepackage[utf8]{inputenc}
\usepackage[active]{srcltx}

\DeclareMathOperator{\dist}{dist}
\DeclareMathOperator{\diam}{diam}

\DeclareMathOperator{\esssup}{ess\, sup}

\begin{document}
\newtheorem{oberklasse}{OberKlasse}
\newtheorem{lemma}[oberklasse]{Lemma}
\newtheorem{proposition}[oberklasse]{Proposition}
\newtheorem{definition}[oberklasse]{Definition}
\newtheorem{remark}[oberklasse]{Remark}

\newcommand{\R}{\ensuremath{\mathbbm{R}}}
\newcommand{\N}{\ensuremath{\mathbbm{N}}}
\newcommand{\Z}{\ensuremath{\mathbbm{Z}}}
\newcommand{\mc}{\mathcal}
\newcommand{\eps}{\ensuremath{\varepsilon}}

\title{Provably convergent implementations of the subdivision algorithm for the computation of invariant objects}
\author{Janosch Rieger\footnote{Monash University, School of Mathematical Sciences,
9 Rainforest Walk, Victoria 3800, Australia}}
\date{\today}
\maketitle

\begin{abstract}
The subdivision algorithm by Dellnitz and Hohmann for the computation of invariant 
sets of dynamical systems decomposes the relevant region of the state space into boxes
and analyzes the induced box dynamics.
Its convergence is proved in an idealized setting, assuming 
that the exact time evolution of these boxes can be computed.

In the present article, we show that slightly modified, directly implementable
versions of the original algorithm are convergent under very mild assumptions 
on the dynamical system.
In particular, we demonastrate that neither a fine net of sample points nor very
accurate approximations of the precise dynamics are necessary to guarantee
convergence of the overall scheme.

\end{abstract}

{\bf AMS classification numbers:} 37N30, 65L70

\section{Introduction}

The software package GAIO for the computation of invariant objects is based on the subdivision
algorithm, which was proposed in the research article \cite{DH97} in 1997.
The basic idea of this algorithm is to cover the relevant region of the state space with boxes, 
to analyze the dynamics induced on the box cover and to refine the cover successively in such 
a way that a neat overapproximation of the desired object is obtained with manageable computational effort.

Today, twenty years later, the subdivision algorithm and the GAIO package are firmly established.
Their applications include the computation of invariant manifolds of autonomous (see the original article \cite{DH97}) and nonautonomous dynamical systems (see \cite{PR09}), 
the solution of global optimization problems (see \cite{SD06}),
the approximation of Pareto sets in multiobjective optimization (see \cite{DSH05}), 
the computation of optimal stabilizing feedback laws (see \cite{GJ05}),
the approximation of almost invariant sets (see \cite{DJ99}),
the identification of coherent structures in 3d fluid flows (see \cite{FP09}),
the design of space missions (see \cite{DJPT06} and \cite{MR06}),
the computation of rigorous bounds in uncertainty quantification (see \cite{DKZ17}),
and the analysis of the formation of prices in quantitative finance (see \cite{CK16}).
The package is also used in the context of rigorous computations in dynamics (see \cite{Mi02}), 
and, in combination with Monte-Carlo methods, subdivision techniques have been established 
as a major tool in computational molecular dynamics (see \cite{DDJS98} and later work
of the authors).

\medskip

It was noted in \cite{Ju00} that convergence of the subdivision algorithm 
had been proved under idealized conditions.
It is, indeed, assumed in \cite{DH97}, that precise images of boxes in 
state space under the action of the dynamical system can be computed, 
which is not possible in a concrete implementation. 
Therefore, it was proposed in \cite{Ju00} in the context of discrete-time 
systems to work with numerical overapproximations of the precise images 
of boxes, which yields rigorous enclosures of the desired invariant objects.
The question, whether the enclosures converge to these objects, remained open.

It became common practice to evolve large sets of sample points in every box 
and to hope that the resulting discrete image induced the correct dynamics 
on the box cover. 
This approach is computationally expensive, in particular for continuous-time 
dynamics, where large ensembles of trajectories are integrated with high 
precision.
At the same time, the strategy is potentially dangerous, because it is 
well-known that invariant sets can react in an extremely sensitive way 
to discretization errors.

\medskip

The aim of the present article is to prove that enclosures of invariant objects,
which are generated by numerical overapproximations in the spirit of \cite{Ju00}, 
do converge.
For discrete as well as for continuous-time systems, we establish sufficient 
conditions for overapproximations to yield an overall convergent subdivision 
algorithm.
We keep these conditions fairly general, hoping that most variants of the 
subdivision algorithm, which are in current use, can be discussed in this framework.

The organization of and the logic behind the sections on discrete-time 
and continuous-time systems is quite similar.
At first, we prove some useful properties of the global relative attractor,
then we discuss the convergence of a global discretization scheme, 
and in the end, we prove that the output of the subdivision algorithm 
is sandwiched between the exact object and the output of the 
global discretization scheme.

\section{Setting and notation}

Let $\R_+$ denote the set of all nonnegative real numbers, 
and consider a vector norm $\|\cdot\|:\R^d\to\R_+$.
For any  $X,Y\subset\R^d$, the quantities
\[\diam(X):=\sup_{x\in X}\sup_{y\in X}\|x-y\|\quad\text{and}\quad
\dist(X,Y):=\sup_{x\in X}\inf_{y\in Y}\|x-y\|\]
are called the diameter of $X$ and the semidistance between $X$ and $Y$.

\medskip

For the spatial discretization of a dynamical system on a given 
compact set $Q\subset\R^d$, we introduce the notion of a cover.

\begin{definition}
Given $\rho>0$, a collection $\Omega=\{D_i\subset\R^d: i\in I\}$ is called
a $\rho$-cover of $Q$ if the index set $I$ is finite, if $Q=\cup_{i\in I}D_i$, 
if $D_i\neq\emptyset$ for all $i\in I$
and if $\diam(D_i)\le\rho$ for all $i\in I$.
\end{definition}

For any index set $I$, the symbol $2^I$ will represent the collection of all 
subsets of $I$, including the empty set.

The definition of nested covers formalizes the idea of a successively
refined sequence of discretizations.

\begin{definition}
Let $(\rho_n)_{n=0}^\infty$ be a sequence with $\rho_n\searrow 0$ as $n\to\infty$.
Then a sequence $(\Omega_n)_{n=0}^\infty$ of $\rho_n$-covers given by 
$\Omega_n=\{D_i^n\subset\R^d: i\in I_n\}$ 
is called nested if for every $n\in N$ and $i\in I_{n+1}$, there exists $j\in I_n$ 
such that $D_i^{n+1}\subset D_j^n$.
\end{definition}

The paper is essentially self-contained.
The only external resources that are used are Theorems 0.3.4 and 1.4.1 
from \cite{AC84}, which summarize well-known arguments from the proof 
of the Cauchy-Peano theorem we do not wish to repeat explicitly.

\section{Discrete-time dynamics}

Consider an autonomous dynamical system 
\begin{equation} \label{deq}
x_{k+1}=f(k_n),\quad k\in\N,
\end{equation}
induced by a homeomorphism $f:\R^d\to\R^d$.
We will be interested in the dynamics near a compact set $Q\subset\R^d$.
The object we wish to approximate is the global relative attractor. 

\begin{definition}
The global attractor of the dynamical system \eqref{deq} relative to $Q$ is the set
\begin{equation} \label{def:AQ:deq}
A_Q:=\cap_{k\in\N}f^k(Q).
\end{equation}
\end{definition}

Some characteristics of this attractor are immediate consequences of its definition.

\begin{lemma}\label{disc:compact}
The set $A_Q$ is compact and has the following properties.
\begin{itemize}
\item [(a)] We have $A_Q=\{x\in\R^d:f^{-k}(x)\in Q\ \forall\,k\in\N\}$.
\item [(b)] The inclusion  $f^{-k}(A_Q)\subset A_Q$ holds for all $k\in\N$.
\item [(c)] If $\tilde Q\subset Q$ is compact and $A_Q\subset\tilde Q$, 
then $A_Q=A_{\tilde Q}$.
\end{itemize}
\end{lemma}
\begin{proof}
Since $f$ is continuous, the set $A_Q$ is an intersection of compact sets 
and hence compact.
Statement (a) is a reformulation of the definition \eqref{def:AQ:deq}.

(b) Let $k\in\N$ and $x\in A_Q$ be given.
By part (a), whe have
\[f^{-l}(f^{-k}(x))=f^{-(k+l)}(x)\in Q\quad\forall l\in\N,\]
so again by part (a), we have $f^{-k}(x)\in A_Q$.

(c) Since $\tilde Q\subset Q$, we have
\[A_{\tilde Q}=\cap_{k\in\N}f^k(\tilde Q)\subset\cap_{k\in\N}f^k(Q)=A_Q.\]
Now let $x\in A_Q$. By part (b), we have
\[f^{-k}(x)\in A_Q\subset\tilde Q\quad\forall\,k\in\N,\]
so part (a) applied to $\tilde Q$ instead of $Q$ implies $x\in A_{\tilde Q}$.
\end{proof}

\subsection{A global discretization scheme}

We overapproximate the dynamics that the mapping $f^{-1}$ induces 
on a cover $\Omega$ of $Q$ by a multivalued mapping $\varphi$ 
acting on its index set $I$.
This abstraction provides a clear picture of the principles 
that are at work and leaves a lot of freedom for the design of concrete 
implementations.
Let us define the discrete analog of $A_Q$.

\begin{definition}
Let $\Omega=\{D_i\subset\R^d: i\in I\}$ be a $\rho$-cover of $Q$,
and let $\varphi:I\to 2^I$ be a mapping.
Then we define 
\[I_\Omega^\varphi:=\{i\in I: \varphi^k(i)\neq\emptyset\ \forall\,k\in\N\}\]
and call $A_\Omega^\varphi:=\cup_{i\in I_\Omega^\varphi}D_i$
the discrete global attractor of $(\Omega,\varphi)$.
\end{definition}

The index set $I_\Omega^\varphi$ and hence $A_\Omega^\varphi$
can be computed using Algorithm \ref{discrete:global:algorithm}, 
which can be considered a variant of Dijkstra's algorithm.
In the following, we show that the discrete attractors computed by this algorithm 
converge to $A_Q$ as $\rho$ tends to zero.
In a first step, we show that the exact global attractor $A_Q$ is always contained in
the numerical approximation.

\begin{proposition} \label{lower}
If $\Omega=\{D_i\subset\R^d: i\in I\}$ is a $\rho$-cover of $Q$,
and a mapping $\varphi:I\to 2^I$ satisfies
\begin{equation} \label{contained}
\big(f^{-1}(D_i)\cap Q\big)\subset\big(\cup_{j\in\varphi(i)}D_j\big)\quad\forall i\in I,
\end{equation}
then $A_Q\subset A_\Omega^\varphi$.
\end{proposition}

\begin{proof}
Fix $x\in A_Q$.
Since $\Omega$ is a cover of $Q$, there exists $i_0\in I$ such that $x\in D_{i_0}$.
Let us prove by induction that for all $k\in\N$,
\begin{equation} \label{IH}
\text{there exists}\ i_k\in\varphi^k(i_0)\ \text{with}\ f^{-k}(x)\in D_{i_k}.
\end{equation}
Since $\varphi^0(i_0)=\{i_0\}$ and $f^0(x)=x$, the statement is true for $k=0$.
Now assume that statement \eqref{IH} is true for some $k\in\N$.
By Lemma \ref{disc:compact} part (a), we have $f^{-(k+1)}(x)\in Q$, 
and since, in addition, $f^{-(k+1)}(x)\in f^{-1}(D_{i_k})$,
it follows from condition \eqref{contained} that 
\[f^{-(k+1)}(x)\in\big(\cup_{j\in\varphi(i_k)}D_j\big).\]
Hence there exists $i_{k+1}\in\varphi(i_k)\subset\varphi^{k+1}(i_0)$ 
with $f^{-(k+1)}(x)\in D_{i_{k+1}}$.
By induction, statement \eqref{IH} holds.
In particular, $\varphi^k(i_0)\neq\emptyset$ for all $k\in\N$,
and hence $i_0\in I^\varphi_\Omega$ and $x\in A^\varphi_\Omega$.
\end{proof}

Now we prove that the numerical approximation $A^\varphi_\Omega$ shrinks 
down to $A_Q$ if the cover of $Q$ and the discrete mapping $\varphi$ 
are refined appropriately.

\begin{proposition} \label{upper}
If $(\rho_n)_n\subset\R_+$ is a sequence with $\lim_{n\to\infty}\rho_n=0$,
if the collections $\Omega_n=\{D_i^n\subset\R^d: i\in I_n\}$ are $\rho_n$-covers of $Q$,
and if the mappings $\varphi_n:I_n\to 2^{I_n}$ satisfy
\begin{equation} \label{not:too:far}
\lim_{n\to\infty}\sup_{i\in I_n}\,\dist(\cup_{j\in\varphi_n(i)}D_j^n,f^{-1}(D_i^n))=0,
\end{equation}
then $\lim_{n\to\infty}\dist(A_{\Omega_n}^{\varphi_n},A_Q)=0$.
\end{proposition}

\begin{proof}
If the statement is false, then there exist $\eps>0$ and $i_0^n\in I^{\rho_n}_{\Omega_n}$,
$n\in\N$, such that, after passing to a subsequence, we have
\[\dist(D^n_{i_0^n},A_Q)\ge\eps\quad\forall\,n\in\N.\]
By definition of $I^{\rho_n}_{\Omega_n}$, for all $n\in\N$
there exist sequences $(i^n_k)_{k=1}^\infty\subset I_n$ with 
\[i^n_{k+1}\in\varphi_n(i^n_k)\quad\forall\,k\in\N.\]
Pick arbitrary points $x^n_k\in D^n_{i^n_k}$ for $k,n\in\N$
and split
\begin{align*}
\|x^n_{k+1}-f^{-1}(x^n_k)\|
\le&\dist(x^n_{k+1},\cup_{j\in\varphi_n(i^n_k)}D_j^n)\\
&+\dist(\cup_{j\in\varphi_n(i^n_k)}D_j^n,f^{-1}(D^n_{i^n_k}))
+\dist(f^{-1}(D^n_{i^n_k}),f^{-1}(x^n_k)).
\end{align*}
Because of $x^n_{k+1}\in\cup_{j\in\varphi_n(i^n_k)}D_j^n$, by \eqref{not:too:far},
since $f^{-1}$ is uniformly continuous on $Q$ and since 
$\dist(D^n_{i^n_k},x^n_k)\le\rho_n$, we conclude that
\begin{equation} \label{almost:trajectory}
\|x^n_{k+1}-f^{-1}(x^n_k)\|\to 0\quad\text{as}\quad n\to\infty.
\end{equation}
Now we construct a trajectory of \eqref{deq} using induction on $k$.
Since $Q$ is compact, we may pass to a subsequence to obtain
\[x_0:=\lim_{n\to\infty}x^n_{i^n_0}\in Q.\]
If points $(x_j)_{j=0}^k\subset Q$ with 
\begin{align}
&x_j=\lim_{n\to\infty}x^n_{i_j^n},\quad j=0,\ldots,k,\label{mylimit}\\
&x_{j+1}=f^{-1}(x_j),\quad j=0,\ldots,k-1,
\end{align}
have been constructed, we may pass to a subsequence again to obtain
\begin{equation} \label{newlimit}
x_{k+1}:=\lim_{n\to\infty}x^n_{i^n_{k+1}}\in Q.
\end{equation}
By \eqref{almost:trajectory}, \eqref{mylimit}, \eqref{newlimit} and continuity of $f^{-1}$,
it follows that
\begin{align*}
\|x_{k+1}-f^{-1}(x_k)\| 
\le& \|x_{k+1}-x^n_{i^n_{k+1}}\|
+ \|x^n_{i^n_{k+1}}-f^{-1}(x^n_{i^n_k})\|\\
&+ \|f^{-1}(x^n_{i^n_k})-f^{-1}(x_k)\|
\to 0\ \text{as}\ n\to\infty,
\end{align*}
so $x_{k+1}=f^{-1}(x_k)$.
By induction, we obtain a sequence $(x_k)_{k=0}^\infty\subset Q$ with
\[x_{k+1}=f^{-1}(x_k),\quad k\in\N.\]
Lemma \ref{disc:compact}(a) yields $x_0\in A_Q$.
This, however, is impossible, since
\begin{align*}
\dist(x_0,A_Q) 
= \dist(\lim_{n\to\infty}x^n_{i^n_0},A_Q)
= \lim_{n\to\infty}\dist(x^n_{i^n_0},A_Q)\ge\eps.
\end{align*}
\end{proof}

\subsection{A subdivision scheme}

Let us prove that the index sets $J_n$ generated by 
Algorithm \ref{discrete:subdivision:algorithm} encode subsets 
$\cup_{j\in J_n}D^n_j\subset Q$ converging to $A_Q$.

\begin{proposition} \label{sub:div:prop}
Let $(\rho_n)_{n=0}^\infty$ be a sequence with $\rho_n\searrow 0$ as $n\to\infty$,
and let $(\Omega_n)_{n=0}^\infty$ given by $\Omega_n=\{D_i^n\subset\R^d: i\in I_n\}$ 
be a nested sequence of $\rho_n$-covers of $Q$.
If the mappings $\varphi_n:I_n\to 2^{I_n}$ satisfy
\begin{equation} \label{n:contained}
\big(f^{-1}(D_i^n)\cap Q\big)\subset\big(\cup_{j\in\varphi_n(i)}D_j^n\big)
\quad\forall\,i\in I_n,\ n\in\N
\end{equation}
and condition \eqref{not:too:far}, then the index sets $J_n$ 
computed by Algorithm \ref{discrete:subdivision:algorithm} satisfy
\begin{equation} \label{sandwiched}
A_Q\subset\big(\cup_{j\in J_n}D^n_j\big)\subset A^{\varphi_n}_{\Omega_n}.
\end{equation}
In particular, the sets $\cup_{j\in J_n}D^n_j$ converge to $A_Q$ from above.
\end{proposition}

\begin{proof}
By definition of the discrete global attractor, by construction of the 
index set $J_n$, and since $J_n^+\subset I_n$, the inclusion 
$\big(\cup_{j\in J_n}D^n_j\big)\subset A^{\varphi_n}_{\Omega_n}$
is correct for any $n\in\N$.

Since $J_0=I^{\varphi_0}_{\Omega_0}$, it follows from condition \eqref{n:contained}
and Proposition \ref{lower} that
\[A_Q\subset\big(\cup_{j\in J_0}D^0_j\big)=A^{\varphi_0}_{\Omega_0}.\]
Assume that \eqref{sandwiched} holds for some $n\in\N$.
Then $\tilde\Omega:=\{D_j^{n+1}:j\in J_{n+1}^+\}$ is a $\rho_{n+1}$-cover of the set 
$\tilde Q:=\cup_{j\in J_n}D^n_j$, and condition \eqref{n:contained}
implies that 
\[\big(f^{-1}(D_i^{n+1})\cap\tilde Q\big)\subset\big(\cup_{j\in\varphi_n(i)}D_j^{n+1}\big)
\quad\forall\,i\in J_{n+1}^+.\]
Thus Lemma \ref{disc:compact} part (c) and Proposition \ref{lower} applied to $\tilde Q$
and $\varphi_{n+1}$ restricted to $J_{n+1}^+$ yield
\[A_Q=A_{\tilde Q}\subset\big(\cup_{j\in J_{n+1}}D^{n+1}_j\big).\]
By induction, inclusion \eqref{sandwiched} holds for all $n\in\N$.

In view of Propositions \ref{lower} and \ref{upper}, 
the sets $\cup_{j\in J_n}D^n_j$ converge to $A_Q$ from above.
\end{proof}


\begin{algorithm}[p]
\DontPrintSemicolon
\KwIn{index set $I$, mapping $\varphi$}
\KwOut{index set $I_\Omega^\varphi$}

\medskip

\tcc{initialize removed, critical and noncritical indices}
$J_r\gets\{i\in I:\varphi(i)=\emptyset\}$\;
$J_c\gets\{i\in I: \varphi(i)\cap J_r\neq\emptyset\}$\;
$J_{nc}\gets I\setminus(J_r\cup J_c)$\;

\medskip

\tcc{remove indices $i$ with $\varphi^k(i)=\emptyset$ for some $k$ recursively}
\While{$\exists i\in J_c: \varphi(i)\cap(J_c\cup J_{nc})=\emptyset$}{
$J_c\gets J_c\setminus\{i\}$\; 
$J_r\gets J_r\cup\{i\}$\;
\For{$j\in J_{nc}$}{
\If{$i\in\varphi(j)$}{
$J_{nc}\gets J_{nc}\setminus\{j\}$\;
$J_c\gets J_c\cup\{j\}$\;
}}}

\Return{$J_c\cup J_{nc}$}\;
\caption{A global scheme for the approximation of $A_Q$. \label{discrete:global:algorithm}}
\end{algorithm}

\begin{algorithm}[p]
\DontPrintSemicolon
\KwIn{mappings $\varphi_n$, nested $\Omega_n=\{D_i^n\subset\R^d: i\in I_n\}$, $n\in\N$}
\KwOut{index set $J_n$}

\medskip

$J_0^+\gets I_0$\;
\For{$n\gets 0$ \textbf{\textup{to}} $\infty$}{
\tcc{compute discrete attractor at current level}
$J_n\gets$ Algorithm \ref{discrete:global:algorithm} applied to $(J_n^+,\varphi_n)$\;
\textbf{break} upon user request\;
\tcc{refine cover of current discrete attractor}
$J_{n+1}^+\gets\{i\in I_{n+1}:D_i^{n+1}\subset D_j^n\ \text{for some}\ j\in J_n\}$\;
}

\Return{$J_n$}\;
\caption{A subdivision scheme for the approximation of $A_Q$. \label{discrete:subdivision:algorithm}}
\end{algorithm}


\subsection{A provably convergent implementation} \label{sec:imp}

For simplicity, we assume that the mapping 
$f^{-1}$ is $L$-Lipschitz, and we limit ourselves to an implementation 
which does not exploit higher order Taylor terms of $f^{-1}$.
For more elaborate overapproximations, which can be treated in 
a similar way, we refer to \cite{Ju00}.

Let $Q$ be a box, let $\rho_0:=\diam(Q)$,
and set $\Omega_0:=\{D_0^0\}$ with $D_0^0=Q$ and $I_0=\{0\}$.
We define a nested sequence of box covers by induction.
Let $\Omega_n=\{D_i^n:i\in I_n\}$ with index set 
$I_n=\{0,\ldots,2^{nd}-1\}$
be a $\rho_n$-cover of $Q$ consisting of boxes $D^n_i$.
Subdivide each $D^n_i$ into $2^d$ commensurate subboxes 
$D^{n+1}_{i2^d},\ldots,D^{n+1}_{(i+1)2^d-1}$, 
define $I_{n+1}:=\{0,\ldots,2^{(n+1)d}-1\}$ and set $\rho_{n+1}:=\frac12\rho_n$.
Then $\Omega_{n+1}=\{D_i^{n+1}:i\in I_{n+1}\}$
is $\rho_{n+1}$-cover of $Q$ consisting of boxes $D^{n+1}_i$.

To construct an overapproximation of $f^{-1}(D^n_i)$ for a given $D^n_i\in\Omega_n$, 
we choose $M\in\N_1$ and decompose each $D_i^n$ into $M^d$ commensurate subboxes 
$E^{i,n}_0,\ldots,E^{i,n}_{M^d-1}$ 
with centers $z^{i,n}_0,\ldots,z^{i,n}_{M^d-1}$
and diameter $\varrho_n:=\rho_n/M$.
Now we define the discrete dynamics by
\begin{equation} \label{phi:for:dis}
\varphi_n(i):=\{j\in I_n: \cup_{l=0}^{M^d-1}\{f^{-1}(z^{i,n}_{l})\}+B_{L\varrho_n}(0)\cap D_j^n\neq\emptyset\}
\end{equation}
and verify that they satisfy our sufficient conditions for convergence 
of the resulting numerical method.

\begin{proposition}
The discrete mappings $\varphi_n$ constructed in \eqref{phi:for:dis} satisfy
conditions \eqref{not:too:far} and \eqref{n:contained}.
\end{proposition}

\begin{proof}
For every $x\in D^n_i$ and $D^n_i\in\Omega_n$ , there exists some 
$l\in\{0,\ldots,M^d-1\}$ with $x\in E^{i,n}_l$.
Since 
\[\|f^{-1}(x)-f^{-1}(z^{i,n}_l)\|
\le L\|x-z^{i,n}_l\|\le L\varrho_n,\]
condition \eqref{n:contained} holds.
By definition of the mappings $\varphi_n$, we have
\begin{align*} 
&\lim_{n\to\infty}\sup_{i\in I_n}\,\dist(\cup_{j\in\varphi_n(i)}D_j^n,f^{-1}(D_i^n))
\le\lim_{n\to\infty}(L+1)\varrho_n=0,
\end{align*}
which is \eqref{not:too:far}.
\end{proof}

Taking $M=1$ in the above, we obtain a provably convergent 
algorithm that only needs a single evaluation of $f^{-1}$ per box.

\section{Continuous-time dynamics}

Consider an autonomous ordinary differential equation
\begin{equation}\label{ode}
\dot x(t)=g(x(t)),\quad t\in\R,
\end{equation} 
with continuous right-hand side $g:\R^d\to\R^d$.
Again, we will only be interested in the dynamics near a compact set $Q\subset\R^d$, 
and we assume throughout this section that the solution map $\phi:\R\times\R^d\to\R^d$ 
of \eqref{ode} is well-defined and continuous.

\begin{definition}
The global attractor relative to $Q$ is the set
\begin{equation} \label{def:AQ}
A_Q:=\cap_{t\in\R_+}\phi(t,Q).
\end{equation}
\end{definition}

As in the discrete-time case, some characteristics of the attractor are immediate.

\begin{lemma}\label{compact}
The set $A_Q$ is compact and has the following properties.
\begin{itemize}
\item [(a)] We have $A_Q=\{x\in\R^d:\phi(-t,x)\in Q\ \forall\,t\in\R_+\}$.
\item [(b)] We have $\phi(-t,A_Q)\subset A_Q\quad\forall\,t\in\R_+$.
\item [(c)] If $\tilde Q\subset Q$ is compact and $A_Q\subset\tilde Q$,
then $A_Q=A_{\tilde Q}$.
\end{itemize}
\end{lemma}
\begin{proof}
For any $t\in\R_+$, the mapping $x\mapsto\phi(t,x)$ is continuous,  
so $\phi(t,Q)$ is compact, and hence the intersection $A_Q$ of these sets
is compact as well. 
Statement (a) is a reformulation of the definition given in \eqref{def:AQ}.

(b) Let $t\in\R_+$ and $x\in A_Q$ be given.
By part (a), we have
\[\phi(-s,\phi(-t,x))=\phi(-(s+t),x)\in Q\quad\forall s\in\R_+,\]
so that $\phi(-t,x)\in A_Q$ follows from part (a).

(c) It is clear that
\[A_{\tilde Q}=\cap_{t\in\R_+}\phi(t,\tilde Q)\subset\cap_{t\in\R_+}\phi(t,Q)=A_Q.\]
Now let $x\in A_Q$. By part (b), we have
\[\phi(-t,x)\in A_Q\subset\tilde Q\quad\forall\,t\in\R_+,\]
so part (a) implies $x\in A_{\tilde Q}$.
\end{proof}

\subsection{A global discretization scheme}

We reduce the continuous-time dynamics to a discrete-time system
by considering the time-$h$ map $f(x):=\phi(h,x)$ with inverse 
$f^{-1}(x)=\phi(-h,x)$.
The discrete global attractors are then computed by applying 
Algorithm \ref{discrete:global:algorithm} to covers $\Omega$ and 
discrete maps $\varphi$ as in the discrete-time case.
The following result and its proof are identical with Proposition \ref{lower}
up to notation.

\begin{proposition} \label{ct:lower}
Fix an arbitrary $h>0$.
If $\Omega=\{D_i\subset\R^d: i\in I\}$ is a $\rho$-cover of $Q$,
and if the mapping $\varphi:I\to 2^I$ satisfies
\begin{equation} \label{ct:contained}
\big(\phi(-h,D_i)\cap Q\big)\subset\big(\cup_{j\in\varphi(i)}D_j\big)\quad\forall i\in I,
\end{equation}
then $A_Q\subset A_\Omega^{\varphi}$.
\end{proposition}

Now we prove an analog of Proposition \ref{upper}.

\begin{proposition} \label{ct:upper}
Let $(h_n)_{n=0}^\infty,(\rho_n)_{n=0}^\infty\subset\R_+$ be sequences with 
\[\lim_{n\to\infty}h_n=\lim_{n\to\infty}\rho_n=0.\]
If $\Omega_n=\{D_i^n\subset\R^d: i\in I_n\}$ are $\rho_n$-covers of $Q$,
and if $\varphi_n:I_n\to 2^{I_n}$ are mappings such that
\begin{align} 
&\lim_{n\to\infty}\sup_{i\in I_n}\sup_{j\in\varphi_n(i)}\dist(D_j^n,D_i^n)=0,
\label{not:too:far:1}\\
&\lim_{n\to\infty}\sup_{i\in I_n}\sup_{j\in\varphi_n(i)}\sup_{x\in D^n_j}\sup_{z\in D^n_i}
\|h_n^{-1}(x-z)+g(z)\|=0,\label{not:too:far:2}
\end{align}
then $\lim_{n\to\infty}\dist(A_{\Omega_n}^{\varphi_n},A_Q)=0$.
\end{proposition}

\begin{proof}
If this statement is false, then there exist $\eps>0$
and $i^n_0\in I^{\varphi_n}_{\Omega_n}$, $n\in\N$, such that,
after passing to a subsequence, we have
\[\dist(D^n_{i^n_0},A_Q)\ge\eps\quad\forall\,n\in\N.\]
By definition of $I_{\Omega_n}^{\varphi_n}$, for every $n\in\N$, 
there exists  
$(i_k^n)_{k=1}^\infty\subset I_n$ with
\[i^n_{k+1}\in\varphi(i^n_k)\quad \forall\,k\in\N.\]
Fix arbitrary $x^n_k\in D^n_{i^n_k}$ for all $k,n\in\N$ and define continuous
piecewise linear functions $y_n:(-\infty,0]\to\R^d$ by
\[y_n(t):=\frac{t+(k+1)h_n}{h_n}x^n_k-\frac{t+kh_n}{h_n}x^n_{k+1}\quad\forall
t\in[-(k+1)h_n,-kh_n],\ k\in\N.\]
By condition \eqref{not:too:far:1}, and since for all $k\in\N$, we have
\[\|y_n(t)-x^n_k\|\le\|x^n_{k+1}-x^n_k\|\le\dist(D_j^n,D_i^n)+\rho_n
\quad\forall\,t\in[-(k+1)h_n,-kh_n],\]
these functions satisfy
\begin{equation} \label{close:to:node}
\lim_{n\to\infty}\sup_{k\in\N}\sup_{t\in[-(k+1)h_n,-kh_n]}\|y_n(t)-x^n_k\|=0.
\end{equation}
Because of condition \eqref{not:too:far:2}, we find
\begin{equation}
\left.\begin{aligned}
&\lim_{n\to\infty}\sup_{k\in\N}\sup_{t\in(-(k+1)h_n,-kh_n)}\|y_n'(t)+g(x^n_k)\|\\
&=\lim_{n\to\infty}\sup_{k\in\N}\sup_{t\in(-(k+1)h_n,-kh_n)}
\|h_n^{-1}(x^n_k-x^n_{k+1})+g(x^n_k)\|=0.
\end{aligned}\right\}
\label{space:est}
\end{equation}
Since $\max_{x\in Q}\|g(x)\|<\infty$, we may conclude from this statement that
\begin{equation} \label{der:bdd}
\sup_{n\in\N}\esssup_{t\in(-\infty,0]}\|y_n'(t)\|<\infty,
\end{equation}
and by construction, we have
\begin{equation} \label{f:bdd}
\sup_{n\in\N}\sup_{t\in(-\infty,0]}\|y_n(t)\|\le\max_{x\in Q}\|x\|<\infty.
\end{equation}
Because of estimates \eqref{close:to:node}, \eqref{space:est}, \eqref{der:bdd} 
and \eqref{f:bdd} and Theorems 0.3.4 and 1.4.1 in \cite{AC84},
there exists an absolutely continuous function 
$y:(-\infty,0]\to\R^d$ with
\[y'(t)=g(y(t))\quad\text{for almost every}\ t\in(-\infty,0]\]
and such that, along a subsequence, $y_n(t)\to y(t)$ holds for all $t\in(-\infty,0]$. 
Since $g$ is continuous and $y$ is absolutely continuous, it follows that $y'$ 
possesses a continuous representation and $y$ is a $C^1$ solution of \eqref{ode}.
Moreover, it follows from \eqref{close:to:node} and 
$x^n_k\in Q$ for all $k,n\in\N$ that
\[\lim_{n\to\infty}\sup_{t\in(-\infty,0]}\dist(y_n(t),Q)=0,\]
and since $Q$ is compact, we have
\[y(t)\in Q\quad\forall t\in(-\infty,0].\]
Hence $y(0)\in A_Q$ by Lemma \ref{compact} part (a).
On the other hand, we have
\[\dist(y(0),A_Q)=\lim_{n\to\infty}\dist(y_n(0),A_Q)\ge\eps,\]
which is a contradiction.
\end{proof}

\subsection{A subdivision algorithm}

The discrete-time subdivision algorithm given as 
Algorithm \ref{discrete:subdivision:algorithm} can be applied to nested covers 
$\Omega_n$ and the discrete maps $\varphi_n$ discussed above.  
The proof of the following convergence result is completely analogous to that of 
Proposition \ref{sub:div:prop} with $f^{-1}$ being replaced by $\phi(-h_n,\cdot)$
and Propositions \ref{ct:lower} and \ref{ct:upper} being invoked instead of Propositions
\ref{lower} and \ref{upper}.

\begin{proposition}
Let $(h_n)_{n=0}^\infty,(\rho_n)_{n=0}^\infty\subset\R_+$ be sequences 
with $h_n,\rho_n\searrow 0$ as $n\to\infty$,
and let $(\Omega_n)_{n=0}^\infty$ given by $\Omega_n=\{D_i^n\subset\R^d: i\in I_n\}$ 
be a nested sequence of $\rho_n$-covers of $Q$.
If $\varphi_n:I_n\to 2^{I_n}$ are mappings satisfying
\begin{equation} \label{now:contained}
\big(\phi(-h_n,D_i^n)\cap Q\big)\subset\big(\cup_{j\in\varphi_n(i)}D_j^n\big)
\quad\forall i\in I_n
\end{equation}
and conditions \eqref{not:too:far:1} and \eqref{not:too:far:2}, then the 
index sets $J_n\subset I_n$ computed by Algorithm \ref{discrete:subdivision:algorithm} satisfy
\begin{equation*} 
A_Q\subset\big(\cup_{j\in J_n}D^n_j\big)\subset A^{\varphi_n}_{\Omega_n}.
\end{equation*}
In particular, the sets $\cup_{j\in J_n}D^n_j$ converge to $A_Q$ from above.
\end{proposition}

\subsection{A provably convergent implementation}

For simplicity, assume that $g$ is globally $P$-bounded and $L$-Lipschitz,
i.e. that
\begin{equation} \label{setting}
\|g(x)\|\le P\quad \text{and}\quad\|g(x)-g(z)\|\le L\|x-z\|\quad\forall\,x,z\in\R^d,
\end{equation}
which, in view of the Picard-Lindel\"of theorem and the Gronwall lemma, implies that the standing assumptions of this section hold.

\medskip

Consider the nested sequence $(\Omega_n)_{n=0}^\infty$ of $\rho_n$-covers 
$\Omega_n=\{D_i^n:i\in I_n\}$ of a box $Q$ defined in paragraph \ref{sec:imp},
as well as a decomposition $E^{i,n}_0,\ldots,E^{i,n}_{M^d-1}$ of
each $D_i^n$ into $M^d$ commensurate subboxes 
with centers $z^{i,n}_0,\ldots,z^{i,n}_{M^d-1}$
and diameter $\varrho_n:=\rho_n/M$.
Choose a sequence $(h_n)_{n=0}^\infty\subset\R_+$ of temporal step-sizes with 
$\lim_{n\to\infty}h_n=0$ and $\lim_{n\to\infty}h_n^{-1}\rho_n=0$.

In contrast to the discrete-time case, the mapping $x\mapsto\phi(-h_n,x)$ is not 
explicitly available, so we approximate it by $N$ steps of Euler's scheme with
a finer step-size $\theta_n:=h_n/N$, which is given in terms of the initial 
value $\phi_E(0,x):=x$ and the iteration
\begin{align*}
&\phi_E(-(k+1)\theta_n,x):=\phi_E(-k\theta_n,x)-\theta_n g(\phi_E(-k\theta_n,x)),\quad
k=0,\ldots,N-1.
\end{align*}
We define the discrete overapproximating dynamics by
\begin{equation} \label{cont:time:map}
\varphi_n(i):=\{j\in I_n:\cup_{l=0}^{M^d-1}\{\phi_E(-h_n,z^{i,n}_l)\}+B_{r_n}(0)\cap D^n_j\neq\emptyset\},\quad 
i\in I_n,
\end{equation}
with parameter 
\[r_n:=e^{Lh_n}\varrho_n+\tfrac{1}{2N}Ph_n(e^{Lh_n}-1)\]
and check that this choice induces an overall convergent subdivision algorithm.
We begin by collecting some information on Euler's scheme
in the setting \eqref{setting}.

\begin{lemma}
For any $k\in\N$ and $x,y,z\in\R^d$, Euler's scheme satisfies
the following estimates:
\begin{align}
&\|\phi_E(-h_n,x)-x\|\le Ph_n,\label{E1}\\
&\|\phi_E(-h_n,y)-\phi_E(-h_n,z)\|\le e^{Lh_n}\|y-z\|,\label{E2}\\
&\|\phi(-h_n,x)-\phi_E(-h_n,x)\|\le\tfrac{1}{2N}Ph_n(e^{Lh_n}-1),\label{E3}\\
&\|h_n^{-1}(\phi_E(-h_n,x)-x)+g(x)\|\le\tfrac12 LPh_n.\label{E4}
\end{align}
\end{lemma}

\begin{proof}
First note that for every $x\in\R^d$, $t\in\R_+$ and $k\in\N$, we have bounds
\begin{align}
&\|\phi(-t,x)-x\|\le\int_{-t}^0\|g(s)\|ds\le P|t|,\label{traj:far}\\
&\|\phi_E(-k\theta_n,x)-x\|
\le\theta_n\sum_{j=0}^{k-1}\|g(\phi_E(-j\theta_n,x))\|
\le Pk\theta_n,\label{Euler:far}
\end{align}
which is inequality \eqref{E1}.
For all $y,z\in\R^d$, we have 
\begin{equation}
\left.\begin{aligned}
\|\phi_E(-\theta_n,y)-\phi_E(-\theta_n,z)\|
&=\|(y-\theta_ng(y))-(z-\theta_ng(z))\|\\
&\le(1+L\theta_n)\|y-z\|,
\end{aligned}\right\} \label{Euler:stability}
\end{equation}
and a simple induction yields 
\[\|\phi_E(-k\theta_n,y)-\phi_E(-k\theta_n,z)\|
\le(1+L\theta_n)^k\|y-z\|
\le(1+\tfrac{Lh_n}{N})^N\|y-z\|,\]
which gives inequality \eqref{E2}.
Using inequality \eqref{traj:far}, we obtain the bound
\begin{equation}
\left.\begin{aligned}
&\|\phi(-\theta_n,x)\!-\!\phi_E(-\theta_n,x)\|
=\|\big(x\!-\!\!\int_{-\theta_n}^0\!g(\phi(t,x))dt\big)-\big(x-\theta_ng(x)\big)\|\\
&\le\int_{-\theta_n}^0\|g(\phi(t,x))dt-g(x)\|dt
\le L\int_{-\theta_n}^0\|\phi(t,x)-x\|dt
\le\tfrac12 LP\theta_n^2
\end{aligned}\right\}\label{local:error}
\end{equation}
for the local error of Euler's scheme.
Let us prove the estimate 
\begin{align}
\|\phi(-k\theta_n,x)-\phi_E(-k\theta_n,x)\|
\le\tfrac12 LP\theta_n^2\sum_{j=0}^{k-1}(1+L\theta_n)^j. \label{Euler:global}
\end{align}
by induction.
For $k=0$, the statement is trivially satisfied.
Now assume that \eqref{Euler:global} holds for some $k\in\N$.
Using inequalities \eqref{Euler:stability} and \eqref{local:error}
as well as the induction hypothesis, we obtain
\begin{align*}
&\|\phi(-(k+1)\theta_n,x)-\phi_E(-(k+1)\theta_n,x)\|\\
&\le\|\phi(-\theta_n,\phi(-k\theta_n,x))-\phi_E(-\theta_n,\phi(-k\theta_n,x))\|\\
&\qquad+\|\phi_E(-\theta_n,\phi(-k\theta_n,x))-\phi_E(-\theta_n,\phi_E(-k\theta_n,x))\|\\
&\le\tfrac12 LP\theta_n^2+(1+L\theta_n)\tfrac12 LP\theta_n^2\sum_{j=0}^{k-1}(1+L\theta_n)^j
=\tfrac12 LP\theta_n^2\sum_{j=0}^k(1+L\theta_n)^j,
\end{align*}
so the representation \eqref{Euler:global} of the global error of Euler's scheme
is correct.
Now inequality \eqref{E3} follows from
\begin{align*}
&\|\phi(-h_n,x)-\phi_E(-h_n,x)\|
\le\tfrac12 LP\theta_n^2\sum_{j=0}^{N-1}(1+L\theta_n)^j\\
&=\tfrac12 LP\theta_n^2\frac{(1+L\theta_n)^N-1}{L\theta_n}
=\frac{1}{2N}Ph_n((1+\frac{Lh_n}{N})^N-1)
\le\frac{1}{2N}Ph_n(e^{Lh_n}-1).
\end{align*}
Finally, use inequality \eqref{Euler:far} to estimate
\begin{align*}
&\|(k\theta_n)^{-1}(\phi_E(-k\theta_n,x)-x)+g(x)\|
\le k^{-1}\sum_{j=0}^{k-1}\|g(\phi_E(-j\theta_n,x))-g(x)\|\\
&\le k^{-1}\sum_{j=0}^{k-1}L\|\phi_E(-j\theta_n,x)-x\|
\le k^{-1}LP\theta_n\sum_{j=0}^{k-1}j\le\tfrac12 LPk\theta_n,
\end{align*}
which proves inequality \eqref{E4}.
\end{proof}

Now we prove that for any fixed $N\in\N_1$, the overapproximations 
$\varphi_n$ defined in \eqref{cont:time:map}
 yield an overall convergent subdivision algorithm.

\begin{proposition}
The mappings $\varphi_n$ constructed in \eqref{cont:time:map}
satisfy conditions \eqref{not:too:far:1}, \eqref{not:too:far:2}
and \eqref{now:contained}.
\end{proposition}

\begin{proof}
We use estimate \eqref{E1} to compute
\begin{align*}
&\lim_{n\to\infty}\sup_{i\in I_n}\sup_{j\in\varphi_n(i)}\dist(D_j^n,D_i^n)\\
&\le\lim_{n\to\infty}\sup_{i\in I_n}\sup_{j\in\varphi_n(i)}
\big\{\dist(D_j^n,\cup_{l=0}^{M^d-1}\{\phi_E(-h_n,z^{i,n}_l)\})\\
&\qquad+\max_{l=0,\ldots,M^d-1}\|\phi_E(-h_n,z^{i,n}_l)-z^{i,n}_l)\|\\
&\le\lim_{n\to\infty}(\rho_n+r_n+Ph_n)=0,
\end{align*}
which is condition \eqref{not:too:far:1}.
Using estimate \eqref{E4}, we obtain for arbitrary $n\in\N$, $i\in I_n$, 
$j\in\varphi_n(i)$, $x\in D^n_j$ and $z\in D^n_i$ a uniform bound
\begin{align*}
&\|h_n^{-1}(x-z)+g(z)\|\\
&\le\inf_{y\in D^n_i} h_n^{-1}\|x-\phi_E(-h_n,y)\|
+\sup_{y\in D^n_i}\|h_n^{-1}(\phi_E(-h_n,y)-y)+g(y)\|\\
&\qquad+\sup_{y\in D^n_i}h_n^{-1}\|y-z\|+\sup_{y\in D^n_i}\|g(z)-g(y)\|\\
&\le h_n^{-1}(\rho_n+r_n)+\tfrac12LPh_n+h_n^{-1}\rho_n+L\rho_n,
\end{align*}
so \eqref{not:too:far:2} holds.
Finally, estimates \eqref{E2} and \eqref{E3} imply for any $x\in D^n_i$ that
\begin{align*}
&\dist(\phi(-h_n,x),\cup_{l=0}^{M^d-1}\{\phi_E(-h_n,z^{i,n}_l)\}\|\\
&\le|\phi(-h_n,x)-\phi_E(-h_n,x)\|
+\dist(\phi_E(-h_n,x),\cup_{l=0}^{M^d-1}\{\phi_E(-h_n,z^{i,n}_l)\})\\
&\le\tfrac{1}{2N}Ph_n(e^{Lh_n}-1)+e^{Lh_n}\dist(x,\cup_{l=0}^{M^d-1}\{z^{i,n}_l\})
\le r_n,
\end{align*}
which proves \eqref{now:contained}.
\end{proof}

Taking $M=1$ and $N=1$ in the above, we obtain a provably convergent 
algorithm that only needs a single Euler step per box.

\section{Conclusion}

We proved that modifications of the subdivision algorithm in the spirit of 
the paper \cite{Ju00} are convergent and verified the usefulness
of this approach. 
It turns out that the overall algorithm is so robust that 
a single evaluation of the dynamical system per box in discrete
time and a single Euler step per box in continuous time suffice to generate 
a convergent numerical method.

\medskip

It would be desirable to know which choices of discretizations and parameters 
yield optimal performance of the implementations we presented.
The basic tradeoff is easy to understand.
Accurate overapproximations of the exact dynamics make the construction of the
mappings $\varphi_n$ very costly, but accelerate the graph theoretical part
of the algorithm and may also lead to a reduction of complexity by eliminating 
larger irrelevant regions at a coarse level. 
As it is, even under idealized conditions, very difficult to prove quantitative results
about approximations of invariant sets, we believe that the question, how these 
two effects can be balanced in an optimal way, cannot be expected 
to be answered rigorously.

\end{document}